\documentclass[11pt]{article} 
\usepackage[utf8]{inputenc} 

\usepackage[margin=1.25in]{geometry} 
\usepackage{graphicx} 

\usepackage{booktabs} 
\usepackage{array} 
\usepackage{paralist} 
\usepackage{verbatim} 
\usepackage{mathrsfs} 
\usepackage{amssymb}
\usepackage{amsthm}
\usepackage{amsmath,amsfonts,amssymb}
\usepackage{esint}
\usepackage{graphics}
\usepackage{enumerate}
\usepackage{mathtools}
\usepackage{xfrac}
\usepackage{bbm}
\usepackage{subcaption}
\usepackage{stmaryrd} 

\let\oldsquare\square 

\usepackage{mathabx}

\renewcommand{\square}{\oldsquare}

\usepackage{tocloft}

\usepackage[usenames,dvipsnames]{xcolor}
\usepackage[colorlinks=true, pdfstartview=FitV, linkcolor=blue, citecolor=blue, urlcolor=blue]{hyperref}
\usepackage[normalem]{ulem}

\newcommand{\la}{\lambda}

\numberwithin{equation}{section}

\newtheorem{theorem}{Theorem}[section]

\newtheorem{lemma}[theorem]{Lemma}
\theoremstyle{definition}

\newtheorem{remark}[theorem]{Remark}

\let\originalleft\left
\let\originalright\right
\renewcommand{\left}{\mathopen{}\mathclose\bgroup\originalleft}
\renewcommand{\right}{\aftergroup\egroup\originalright}

\newcommand{\vertiii}{\vert\kern-0.3ex\vert\kern-0.25ex\vert}

\newcommand*{\N}{\ensuremath{\mathbb{N}}}

\newcommand*{\R}{\ensuremath{\mathbb{R}}}
\newcommand*{\Zd}{\ensuremath{\mathbb{Z}^d}}
\newcommand*{\Rd}{\ensuremath{\mathbb{R}^d}}

\DeclareSymbolFont{boldoperators}{OT1}{cmr}{bx}{n}
\SetSymbolFont{boldoperators}{bold}{OT1}{cmr}{bx}{n}
\renewcommand{\a}{\mathbf{a}}

\newcommand{\ahom}{\bar{\a}}

\newcommand{\cu}{\square}

\renewcommand{\P}{\mathbb{P}}
\newcommand{\E}{\mathbb{E}}
\newcommand{\X}{\mathcal{X}}

\renewcommand{\O}{\mathcal{O}}

\newcommand{\uhom}{\overline{u}}

\newcommand{\indc}{1}

\DeclareMathOperator{\dist}{dist}

\DeclareMathOperator{\size}{sl}

\newcommand{\e}{\varepsilon}

\newcommand{\avsum}{\mathop{\mathpalette\avsuminner\relax}\displaylimits}

\makeatletter
\newcommand\avsuminner[2]{%
	{\sbox0{$\m@th#1\sum$}%
		\vphantom{\usebox0}%
		\ooalign{%
			\hidewidth
			\smash{\,\rule[.23em]{8.8pt}{1.1pt} \relax}%
			\hidewidth\cr
			$\m@th#1\sum$\cr
		}%
	}%
}
\makeatother

\makeatletter
\newcommand\avsuminnerr[2]{%
	{\sbox0{$\m@th#1\sum$}%
		\vphantom{\usebox0}%
		\ooalign{%
			\hidewidth
			\smash{\,\rule[.23em]{6pt}{0.7pt} \relax}%
			\hidewidth\cr
			$\m@th#1\sum$\cr
		}%
	}%
}
\makeatother

\def\Xint#1{\mathchoice
	{\XXint\displaystyle\textstyle{#1}}%
	{\XXint\textstyle\scriptstyle{#1}}%
	{\XXint\scriptstyle\scriptscriptstyle{#1}}%
	{\XXint\scriptscriptstyle\scriptscriptstyle{#1}}%
	\!\int}
\def\XXint#1#2#3{{\setbox0=\hbox{$#1{#2#3}{\int}$}
		\vcenter{\hbox{$#2#3$}}\kern-.5\wd0}}

\def\fint{\Xint-}

\makeatletter 
\newcommand{\negphantom}{\v@true\h@true\negph@nt} 
\newcommand{\neghphantom}{\v@false\h@true\negph@nt} 
\newcommand{\negph@nt}{\ifmmode\expandafter\mathpalette 
	\expandafter\mathnegph@nt\else\expandafter\makenegph@nt\fi} 
\newcommand{\makenegph@nt}[1]{%
	\setbox\z@\hbox{\color@begingroup#1\color@endgroup}\finnegph@nt} 
\newcommand{\finnegph@nt}{%
	\setbox\tw@\null 
	\ifv@ \ht\tw@\ht\z@\dp\tw@\dp\z@\fi \ifh@\wd\tw@-\wd\z@\fi\box\tw@} 
\newcommand{\mathnegph@nt}[2]{%
	\setbox\z@\hbox{$\m@th #1{#2}$}\finnegph@nt} 
\makeatother

\usepackage{titlesec}

\newcommand{\addperiod}[1]{#1.}
\titleformat{\section}
{\normalfont\Large}{\thesection.}{0.5em}{}
\titleformat*{\subsection}{\normalfont\large}
\titleformat{\subsubsection}[runin]
{\bfseries}
{\thesubsubsection.}
{0.5em}
{\addperiod}
\titleformat*{\subsubsection}{\bfseries}
\titleformat*{\paragraph}{\bfseries}
\titleformat*{\subparagraph}{\large\bfseries}

\title{\bf \Large Optimal convergence rates for the spectrum of the graph {L}aplacian on {P}oisson point clouds}

\author{
		Scott Armstrong
	\thanks{Courant Institute of Mathematical Sciences, New York University.
		{\footnotesize \href{mailto:scotta@cims.nyu.edu}{scotta@cims.nyu.edu}.}
	}
\and 
	Raghavendra Venkatraman
	\thanks{Courant Institute of Mathematical Sciences, New York University.
		{\footnotesize \href{mailto:raghav@cims.nyu.edu}{raghav@cims.nyu.edu}.}
	}
}
\date{\today} 

\usepackage[nottoc,notlot,notlof]{tocbibind}

\begin{document}
	
	\maketitle
	
	\begin{abstract}
		We prove optimal convergence rates for eigenvalues and eigenvectors of the graph Laplacian on Poisson point clouds. Our results are valid down to the critical percolation threshold, yielding error estimates for relatively sparse graphs. 
	\end{abstract}
	
	\setcounter{tocdepth}{2}  
	
	\section{Introduction}
	\label{s.intro}
	This paper establishes optimal convergence rates for the eigenvalues and eigenvectors of the graph Laplacian towards their continuum counterpart, for a random geometric graph obtained from connecting points that are within unit distance of each other, in an instance of a Poisson point process on~$\Rd, d \geqslant 2$, whose intensity~$\alpha $ is greater than the critical percolation intensity~$\alpha_c(d)$. Our convergence rates on the eigenvalues, and on eigenvectors (in norms that correspond to~$L^2, L^\infty,$ and in~$C^{0,1}$) are optimal (since they match with results in continuum homogenization theory), and are valid down to the percolation threshold (as we explain below). 
	
	\smallskip	
	Convergence rates for the spectrum of the graph Laplacian on random geometric graphs, towards its continuum counterparts, are of interest in diverse applications (see \cite{GGHS,CGL,CGT} and references therein, for applications inspired from machine learning, \cite{BG,ABV}  for a class of mesh-free methods for solving elliptic problems on manifolds, \cite{BS} for applications to consistency of topological data analysis). The thread connecting these papers, and ours, to questions from data science is this: the   effectiveness of data-driven methods for statistical inference problems such as classification is often tied to the so-called \emph{manifold hypothesis}: the emperical belief that real data from application areas are scattered around a low-dimensional \emph{unknown} structure (such as a manifold) embedded in a possibly high dimensional space (see \cite{FMN} for a mathematical understanding of this hypothesis). A fundamental idea underlying this nonlinear dimensionality reduction step is a weighted graph construction on the available point cloud: it serves as a discrete (generally noisy) approximation to the unknown low-dimensional manifold. The spectrum of the resulting graph Laplacian then plays a crucial role in various statistical inference algorithms (such as classification and regression). 
	
	\smallskip
	The question of obtaining rates of convergence for eigenvalues and eigenvectors that are \emph{valid for relatively sparse graphs} has attracted a lot of attention in recent years: see  \cite{BN,Singer,BIK,shi2015convergence,GGHS,Dobson,WR,CGT,CGL,Lu}. Below we will place our results in the context of these papers in detail. Strengthening these results to their optimal form, by using homogenization instead of averaging, is the main contribution of this paper. 
	
\smallskip
Our results are analytic consequences of the quantitative homogenization and large-scale regularity theory for Poisson point clouds developed in our previous paper~\cite{AV2}. That article adapts the quantitative homogenization by ``coarse-graining'' program of \cite{AKM19,AK22} to study the large scale behavior of Poisson point clouds, obtaining sharp estimates on the difference between graph-harmonic functions and harmonic functions. 

\smallskip 

To describe the problem setting informally, let~$\eta$ denote a Poisson point process of intensity~$\alpha $ on~$\Rd$ and construct a random geometric graph by joining points in~$\eta$ that are within unit distance of each other by an edge. It is well known that there exists~$\alpha_c(d) \in (0,\infty)$ such that if~$\alpha > \alpha_c(d)$,\footnote{We largely stick to the notation of our companion paper~\cite{AV2}, with one exception: since we prefer to reserve notations involving~$\la$ for eigenvalues, we denote the intensity of the process in this paper by~$\alpha$ instead of~$\la$ as in~\cite{AV2}.} then the resulting graph has a unique unbounded connected component, denoted in the sequel by~$\eta_*.$ Fixing a convex, or~$C^{1,1}$ domain~$U_0 \subset \Rd,$ let~$\eta_*(U_0)$ denote ``the largest well-connected component of~$\eta_*$ in~$U_0$'' (a precise definition appears below); then, the \emph{graph Laplacian}~$\mathcal{L}u$ of a function~$u:\eta_*(U_0)\to \R$ is given by 
	\begin{equation} \label{e.graphlap}
		\mathcal{L} u (x) := \sum_{y \in \eta_*(U_0)} \a(|y-x|) (u(y) - u(x))\, \ \ x \in \eta_*(U_0)\,.
	\end{equation} 
Here, and elsewhere, for concreteness
\begin{equation*}
	\a(t) := \begin{cases}
		1 & t < 1 \,, \\
		0 & t \geq 0 \,.
	\end{cases}
\end{equation*}
The exact form of~$\a$ is unimportant for our purposes, it can be substantially more general, and we refer to~\cite{AV2} for further discussion on this point.
For any integer~$m$ we set 
\[ U_m := 3^m U_0\,,\]
Given~$f: \eta_*(U_m) \to \R,$ in~\cite{AV2} we consider the problem
\begin{equation} \label{e.GDP}
	\begin{cases}
    -\mathcal{L} u (x) & = f(x)\ \ x \in \eta_*(U_m) \cap \{\mathrm{dist}(x,\partial U_m) \leq 2\} \\
    u(x) &= 0 \ \ x \in  \eta_*(U_m) \cap \{\mathrm{dist}(x,\partial U_m) > 2\}\,,
	\end{cases}
\end{equation}
and show on large scales (that is, for~$m \gg 1$) the solution to~\eqref{e.GDP} are quantitatively approximated by solutions to a \emph{deterministic} constant coefficient PDE on the continuum.
\begin{theorem}[{\cite[Theorem 1.3]{AV2}}]
	\label{t.AV2-main}
	Suppose~$\alpha > \alpha_c(d)$.
	There exists a positive, constant, deterministic matrix~$\ahom$ depending only on~$\alpha$, constants~$C(d,\alpha),s(d) > 0$ 
	and a nonnegative random variable~$\X$ satisfying 
	\begin{equation*}
	\P[ \X \geqslant C t] \leqslant 2 \exp (-t^s)\,, \ \forall t  > 0\,,
	\end{equation*}
	such that, for every~$m \in \N$ with~$3^m > \X$ and every function~$\uhom \in H^2(U_m) \cap H^1_0(U_m)$, the Dirichlet problem 
	\begin{equation} \label{e.Lu=f}
		\left\{
		\begin{aligned}
			& -  \mathcal{L} u = - \nabla \cdot \ahom \nabla \uhom & \mbox{on} & \ \eta_* ( U_m) \cap \{\mathrm{dist}(x,\partial U_m) > 2\}\,,
			\\ &
			u = 0 & \mbox{on} & \ \eta_*(U_m) \cap \{ x \in U_m : \dist(x,U_m) \leqslant 2 \} \,,
		\end{aligned}
		\right.
	\end{equation}
	for the graph Laplacian has a unique solution~$u$ which satisfies, for~$C(d,\alpha,U_0)<\infty$ which depends also on~$U_0$, and we have the estimate
	\begin{equation} \label{e.rateofconv}
		\| u - \uhom \|_{{L}^2(\eta_*(U_m))} 
		\leq C\| \uhom \|_{{L}^2(U_m)} \cdot 
		\begin{cases}
			3^{-m} 
			& \mbox{ if } d \geqslant 3\,, \\
			m^{\sfrac12} 3^{-m}  & \mbox{ if } d = 2 \,.
	\end{cases}	\end{equation}
\end{theorem}

From a numerical analysis perspective, an alternative viewpoint on Theorem~\ref{t.AV2-main} is this: viewed as a numerical method, the problem~\eqref{e.GDP} is a (mesh-free) sort of finite difference scheme that is \emph{consistent} provided we are above the \emph{percolation threshold}~$\alpha > \alpha_c(d).$ We also have the sharp quantitative error estimates provided by~\eqref{e.rateofconv}. We emphasize that the consistency above the percolation threshold is not a mere technicality: in this regime, the graph has significantly fewer connections, and is therefore much \emph{sparser} than that permitted by the theoretical studies quoted above. Notably, in our setting, the graph is not connected per se, but has an unbounded connected component which ``percolates''. Analyses by previous researchers are valid \emph{only in} regimes in which the graph is connected almost surely~(see e.g.\cite{GGHS}), or more commonly, what we will refer to as ``well above graph connectivity''~(see e.g. \cite{CGL,CGT,Dobson,WR,Lu}), in which, the graph has a ``very high intensity'' of connections. 
 
\smallskip
In this paper we provide the optimal answer to the question:  \emph{how many eigenvalues and eigenfunctions of the Dirichlet Laplacian~$-\nabla \cdot \ahom \nabla$ on the domain~$U_0$ are well-approximated, to first order, by (suitably rescaled) eigenvalues and eigenfunctions of the graph Laplacian generated by the above random geometric graph construction? }

\smallskip 
Rather than considering the spectrum of the graph Laplacian on a Euclidean domain (with Dirichlet boundary conditions, as above), the papers quoted above deal with similar point clouds (modeling real data)  clustered on a closed~$d-$dimensional manifold (the ``ground truth'' manifold). Ignoring this distinction for the moment, these papers work at or above the regime of \emph{graph connectivity}. They show (roughly paraphrased to our notation) that with high probability (quantified suitably)
\begin{equation} \label{e.history}
	|3^{2m}\la_{m,k} - \la_{0,k}| \leqslant \begin{cases}
		C  \la_{0,k}^{\sfrac32} 3^{-\frac{m}2}  &  \cite{GGHS} \mbox{ (at graph connectivity)} \\
		C \|\psi_{0,k}\|_{C^3}3^{-m}  & \cite{CGT}  \mbox{ (well above graph connectivity)} \,. 
	\end{cases}
\end{equation}

Here, we have used the notation~$\la_{0,k}$ to denote the~$k$th Dirichlet eigenvalue of the continuum operator~$-\nabla \cdot \ahom \nabla$ \emph{on the domain}~$U_0$ (our stand-in for the ground truth manifold) with associated~$L^2-$normalized eigenfunction~$\psi_{0,k}$; the term~$3^{2m}\la_{m,k}$ denotes the~$k$th eigenvalue of the graph Laplacian where the graph is obtained by joining points that are within \emph{unit distance} on the large domain (or manifold) that has been inflated by a factor of~$3^m$. The results in~\cite{GGHS} obtain (sub-optimal) square-root convergence rates of~$3^{-\frac{m}2},$ provided the intensity of the point cloud is sufficiently high to ensure that the graph is \emph{connected} almost surely. However,  as we explain below, the dependence in~\cite{GGHS} of the prefactor on~$k$, the index of the eigenvalue, is \emph{sharp}, and the constant~$C$ is independent of~$k.$ 

\smallskip 
By contrast, the paper~\cite{CGT} obtains a linear rate of convergence~$3^{-m}$, but is valid only on scales well above the connectivity threshold. In our notation, they require that the intensity~$\alpha$ tends to infinity with the ratio of scale separation. Moreover, the pre-factor in the estimate \emph{depends on}~$k$ through the~$C^3$ norm of the associated~$L^2-$normalized eigenfunction~$\phi_{0,k}.$ In applications inspired by machine learning (as considered by these articles), one does not typically have access to geometric information on the ground truth manifolds, as to the growth rates of their Laplace-Beltrami eigenfunctions. Inserting maximal growth rate estimates (see Lemma~\ref{l.moseriteration} for the case of Dirichlet eigenfunctions for a Euclidean domain, and~\cite{SZ} for a closed manifold)  which asserts that~$\|\psi_{0,k}\|_{L^\infty} \lesssim \la_{0,k}^{\frac{d}{4}}$, the estimate from~\cite{CGT} implies
\[ |3^{2m} \la_{m,k} - \la_{0,k}| \leqslant C 3^{-m} \la_{0,k}^{\frac{d}{4} + \frac32}\,.\] 

\smallskip 
Having obtained convergence rates for eigenvalues, obtaining convergence rates in~$L^2$ for eigenfunctions is straightforward (as will be recalled below), and is essentially given by the eigenvalue rate multiplied by the reciprocal of the spectral gap (the distance to the nearest distinct eigenvalue, see~\eqref{e.sg} below) (see~\cite{GGHS,CGT}). The papers~\cite{CGL,Dobson} also offer rates of convergence for eigenvectors in~$L^\infty,$ provided one is well above the connectivity threshold. 

The summary in~\eqref{e.history} and the subsequent discussion, is not intended to be an exhaustive summary of prior work on quantitative rates of convergence of the spectrum of graph Laplacian towards its continuum counterparts; as mentioned earlier, see also \cite{shi2015convergence,BIK,Lu,Dobson,WR} for other recent contributions to this question with similar results. All of these papers work in a regime that is well above the connectivity threshold, thereby ensuring that the graph is very dense and connected.

\smallskip  

We turn to the first of our main theorem. For its statement, we introduce the spectral gap~$\gamma(\la_{0,k})$ of an eigenvalue~$\la_{0,k}$ via
\begin{equation}
	\label{e.sg}
	\gamma(\la_{0,k}) := \min_j\{|\la_{0,j} - \la_{0,k}| : \la_{0,j} \neq \la_{0,k}\}\,.
\end{equation}
For simplicity, in this paper we focus on the case of the Dirichlet eigenfunctions on Euclidean domains and prove the following statement. 

\begin{theorem} \label{t.main-spec}
Let~$\alpha  > \alpha_c(d)$. There exist constants~$C(d,\alpha) , s(d) > 0$ and a nonnegative random variable~$\X$ satisfying 
\begin{equation} \label{e.stochint}
	\P[ \X > Ct] \leqslant 2 \exp (-t^s) \ \ \mbox{ for all } t > 0\,,
\end{equation}
such that for any~$m \in \N$ with~$3^m > \X,$ and for any~$k \in \N$ such that
\begin{equation}
	\label{e.epskcondition}
	\begin{cases}
		3^{-m}  \sqrt{\la_{0,k}} < \frac1{C} & \mbox{ if } d \geqslant 3 \,, \\
		m^{\sfrac12}3^{-m}  \sqrt{\la_{0,k}}  < \frac1{C} & \mbox{ if } d = 2\,,
	\end{cases}
\end{equation} 
we have 
\begin{equation} \label{e.eigvalrates}
 \frac{	|3^{2m}\la_{m,k} - \la_{0,k}| }{\la_{0,k}}\leqslant \begin{cases}
 	C3^{-m} \sqrt{\la_{0,k}} & \mbox{ if } d \geqslant 3 \,, \\
 		Cm^{\sfrac12}3^{-m} \sqrt{\la_{0,k}}  & \mbox{ if } d = 2\,.
 \end{cases} 
\end{equation}
In addition, for the continuum eigenfunction~$\phi_{0,k}$ \emph{on the domain~$U_m$}, normalized so that 
\[\fint_{U_m} \phi_{0,k}^2 = 1\,,\]
 and the discrete eigenvector~$\phi_{m,k}$ on the graph~$\eta_*(U_m)$, normalized so that
\begin{equation}
	\label{e.weirdnormalization}
	\avsum_{x \in \eta_*(U_m)} \phi_{0,k}(x) \phi_{m,k}(x) = 1\,,
\end{equation} 
 we have the estimate 
\begin{equation} \label{e.evratesL2}
\frac1{\la_{0,k}}	\biggl(\frac{1}{|U_m|}\sum_{x \in \eta_*(U_m)} (\phi_{m,k}(x) - \phi_{0,k} (x)) ^2 \biggr)^{\sfrac12}\leqslant \frac{1}{\gamma(\la_{0,k})} \times  \begin{cases}
		C3^{-m} \la_{0,k}^{\sfrac12} & \mbox{ if } d \geqslant 3\\
		Cm^{\sfrac12}3^{-m} \la_{0,k}^{\sfrac12}  & \mbox{ if } d = 2\,.
	\end{cases} 
\end{equation}
\end{theorem}
Let us highlight the salient features of Theorem~\ref{t.main-spec} and compare it with prior work. 
\begin{itemize}
	\item The condition~\eqref{e.epskcondition} simply refers to the relationship between~$k$ and~$m$ which ensures that the estimates~\eqref{e.eigvalrates}--\eqref{e.evratesL2} contain useful information. The latter estimate is sharp in three aspects: (i) its \emph{linear} scaling dependence on~$3^{-m}$ (contrast between macro and microscales), (ii) dependence of the prefactor on~$\la_{0,k},$ and (iii) validity down to the percolation threshold.  
	
	\item The theorem statement above is rather detailed:  the main point of Theorem~\ref{t.main-spec} is the estimate: \emph{down to the percolation threshold}, that is, if~$3^m > \X,$ where the nonnegative random variable~$\X$ satisfies~\eqref{e.stochint}, then in comparison with previous results~(e.g. as summarized in~\eqref{e.history})
	\begin{equation*}
		|3^{2m} \la_{m,k} - \la_{0,k} | \leqslant C \la_{0,k}^{\sfrac32} 3^{-m}\,. 
	\end{equation*}
	\item The convergence rates in~\eqref{e.eigvalrates} matches with those in periodic homogenization \cite{KLS} (aside from the factor~$\sqrt{\log (3^m)} = C \sqrt{m}$ in~$d=2$ dimensions, which is intrinsic-- see~\cite{AKM19}). 
\end{itemize}
To gain some intuition why the scalings in the preceding theorem are sharp, let us note that viewing eigenfunctions~$\{\phi_{0,k}\}$ as waves with frequency~$\sqrt{\la_{0,k}}$ (or equivalently, wavelength~$\la_{0,k}^{-\sfrac12}$), the condition~\eqref{e.epskcondition} asserts that the length scale of disorder must be smaller the inverse frequency. Therefore, the estimate~\eqref{e.eigvalrates} asserts that the (deterministic) eigenvalue in the continuum approximates the discrete (random) eigenvalue of the graph Laplacian, provided there is enough room for homogenization. 

\smallskip 
Our next theorem concerns convergence rates for the eigenfunctions in a stronger pointwise sense, and convergence of gradients. We obtain these estimates as a consequence of the large-scale regularity we developed in~\cite{AV2}. Unlike in the previous papers~\cite{CGL,Dobson} with similar rates, we can obtain estimates \emph{down to every edge} of the graph, still valid \emph{down to the percolation threshold}. Our theorem also offers an improvement in the dependence of the prefactor on the eigenvalue~$\la_{0,k}$.

\smallskip 
 For the statement, we need some notation. In~\cite{AV2} we constructed a family of first-order correctors~$\{\varphi_{e}: e \in \Rd\}$, which are certain random fields that are stationary modulo additive constants on the unique infinite percolation cluster on all of space,~$\eta_*(\Rd).$ For each~$m$ with~$3^m > \X$, the minimal scale from Theorem~\ref{t.main-spec}, and for each~$k$ satisfying~\eqref{e.epskcondition}, we set 
\begin{equation*}
	v_{m,k} := \phi_{0,k} + \sum_{i=1}^d \partial_{x_i}\phi_{0,k} \varphi_{e_i}\,. 
\end{equation*}

\begin{theorem}
	\label{c.LSR}
Let~$\alpha > \alpha_c(d),$ and let~$\X$ be as in the statement of Theorem~\ref{t.main-spec}. Then, there exists~$C(d,\alpha) > 0$ such that for every~$m \in \N$ such that~$3^m > \X,$ and for every~$k \in \N$ such that~\eqref{e.epskcondition} holds, we have: for any~$x=3^m \zeta \in \eta_*(U_m), \zeta \in U_0$,
\begin{equation} \label{e.Loo}
	|\phi_{m,k} (x) - \phi_{0,k}(x)| \leqslant \frac{C}{\mathrm{dist}(\zeta,\partial U_0)}\bigl( 1 + \sqrt{m}\indc_{d=2} \bigr) (1+\la_{0,k})^{\sfrac{d}4 + \sfrac32} 3^{-m}\,.
\end{equation}
In addition, for any~$x = 3^m\zeta  \in \eta_*(U_m)$ for~$\zeta \in U_0,$ and for any~$y,z \in \eta_*(B_{\X}(x))$ with~$y\sim z$, we have
\begin{align}
	\label{e.C01}
	\lefteqn{ \bigl| \phi_{m,k} (y) - \phi_{m,k}(z) - \bigl( v_{m,k} (y) - v_{m,k}(z) \bigr)\bigr| } 
	\qquad\qquad & 
	\notag \\ & 
	\leqslant \frac{C}{\mathrm{dist}(\zeta,U_0)}  \bigl( 1 + \sqrt{m}\indc_{d=2} \bigr)   (1 + \la_{0,k})^{\sfrac{d}{4} + \sfrac32}  3^{-m}\,. 
\end{align}
\end{theorem}
\begin{remark}
	In contrast with prior results~(see eg. \cite[Theorem 2.3]{CGL}), where the authors prove an ``approximate Lipschitz property'' of eigenfunctions, the estimate~\eqref{e.C01} contains information about \emph{every} edge of the graph. In addition, we also offer improved (and optimal) prefactors in the dependence on~$\la_{0,k}$. To see that this is the optimal rate one can get, we note that by Lemma~\ref{l.moseriteration}, the~$L^\infty$ norm of the eigenfunction~$\phi_{0,k}$ scales like~$\la_{0,k}^{\sfrac{d}4},$ and by the previous theorem,~$\la_{0,k}^{\sfrac32}$ is the prefactor from the convergence rates for eigenvalues (which matches the rate found in periodic homogenization). 
\end{remark}

Numerical experiments in~\cite{CGT} suggest that by using an unweighted graph Laplacian on a random geometric graph on the standard unit sphere~$\mathbb{S}^2$, one obtains at least \emph{second order convergence rates} (that is,~$3^{-2m}$ in the estimate~\eqref{e.eigvalrates}) for eigenvalues and eigenvectors, even for fairly sparse graphs. An analytical understanding of these numerical observations is still lacking. By way of comparison, a centered finite difference scheme based on a 5-point stencil for the approximation of Laplacian eigenfunctions on a planar domain is second order accurate-- this is because of the cancellation of the first order terms coming from~$x \pm e_i, i = 1,2.$ In the present random geometric graph setting, each vertex has a large (and indeed, unbounded) degree, and any such cancellation must necessarily be a large-scale effect. 

\smallskip

Towards gaining an understanding of these observations, and in order to de-emphasize the role played by boundary layers~\cite{SV,MV}, in~\cite{AV1} we considered periodic homogenization for eigenvalues of an operator with a quadratically confining potential, that is, an operator of the form~$-\nabla \cdot \a \nabla  + W,$ where~$\a$ is~$\Zd$-periodic and~$W$ is a smooth potential. Using higher order expansions in homogenization, we showed \emph{quadratic convergence rates} for eigenvalues and eigenfunctions. As we showed, the~$O(\e)$ term vanishes  due to a structural cancellation property of the third order homogenized tensors. 

\smallskip

While we do not pursue this direction here, we expect that a similar explanation of the numerical simulations in~\cite{CGT}, through a higher order corrector expansion, can be shown using the ideas in~\cite{AV1} and the homogenization machinery developed in~\cite{AV2} and the present paper. 
The algebraic computations involving the homogenized tensors are essentially the same in the periodic and random cases, provided that the second-order correctors exist and possess sufficiently good bounds. These may be obtained by the arguments of~\cite{AV2}.
Moreover, as remarked in~\cite{AV2}, adapting the homogenization methods to the setting of general manifolds (as opposed to Euclidean space) does not require fundamentally new ideas.
Finally, we mention that, while the main theorems are stated for the Dirichlet spectrum, similar statements hold for the Neumann spectrum with essentially the same proof.

	  \section*{Acknowledgements}
	  S.A. acknowledges support from the NSF grants DMS-1954357 and DMS-2000200. R.V. acknowledges support from the Simons foundation through award number~733694. R.V. warmly thanks Dejan Slep\v{c}ev for bringing this problem to his attention, and for his encouragement, and thanks Bob Kohn for helpful conversations.  
	  
	\section{Convergence rates for the spectrum of the graph Laplacian}
\subsection{Notation and Preliminaries}
Throughout, we work in dimensions~$d \geqslant 2$, and we largely use the same notation as in our companion paper~\cite{AV2}. 

\smallskip 
\emph{Discrete notation:} A \emph{cube} is a set of the form
\begin{equation*}
	\Zd \cap \bigl(z + [0,N)^d\bigr), \quad z \in \Zd, N \in \N\,. 
\end{equation*}
The center of this cube is~$z$ and its \emph{size} is~$N,$ denoted~$\size(\cu).$  A \emph{triadic cube}~$\cu$ is a cube of the form 
\begin{equation*}
	\cu_m (z) := \Zd \cap \Bigl(z + \Bigl[ - \frac12 3^m, \frac12 3^m \Bigr)^{\!d} \, \Bigr) \,, \quad z \in 3^m \Zd \,, \ m \in \N\,. 
\end{equation*}
When~$z = 0$, we write~$\cu_m$ instead. 

\smallskip 
\emph{Graph construction:}
Let~$\eta$ denote a Poisson point process on~$\Rd$ with intensity~$\alpha > 0.$ We also denote by~$\P$ the law of this random point cloud. For each realization of the point cloud, we create a graph by connecting any pair of points separated by a Euclidean distance no more than one. This graph will have infinitely many connected components, and if~$\alpha > \alpha_c(d)$ (the supercritical regime) the graph has a unique unbounded connected component. 

Given a cube~$\cu,$ \cite[Definition 2.1]{AV2} identifies \emph{well-connected cubes}~$\cu$ characterized by the following properties: the unique connected component of the percolation cluster in~$\cu$ (denoted~$\eta_*(\cu)$) , has diameter at least~$\frac{1}{100} \size(\cu)$, the number of points in~$\cu$ is roughly the volume of the cube, and finally, each point in the cube~$\cu$ has a neighbor in the percolation cluster within distance~$\frac1{100} \size(\cu).$ 

Next, given a Borel set~$U\subset \Rd$, we let~$\eta_*(U)$ denote the largest connected component of~$\eta_*(\cu)\cap U,$ where~$\cu$ is the smallest triadic cube containing~$U$ such that~$3\cu$ is well-connected (in the sense of~\cite[Definition 2.1]{AV2}). Using results of Penrose~\cite{Penrose} and Penrose and Pisztora~\cite{PP}, it is shown in~\cite[Proposition 2.2]{AV2} that large cubes are well-connected except on an event with probability which is exponentially small in the size of the cube. 

\smallskip 
We will use the~$\O_s$ notation introduced in~\cite{AKM19} to measure the stochastic integrability of random variables. Given~$s,\theta > 0,$ and a nonnegative random variable~$\X,$ we write 
\begin{equation*}
	\X = \O_s(\theta) \iff \E \biggl[ \exp \biggl( \frac{X}{\theta} \biggr)^s \biggr] \leqslant 2\,. 
\end{equation*}
By Markov's inequality, it follows that 
\begin{equation*}
	\X = \O_s (\theta) \implies \P [ \X \geqslant \theta t] \leqslant 2 \exp (-t^s), \, \quad t > 0\,.
\end{equation*}
Given a Borel set~$U$ and a function~$f:\eta_*(U) \to \R$, we write 
\begin{equation*}
	(f)_U := \avsum_{x \in \eta_*(U)} f(x) \equiv \frac{1}{|\eta_*(U)|} \sum_{x \in \eta_*(U)} f(x)\,. 
\end{equation*}
Here and in what follows, the notation~$|A|$ denotes the cardinality of the set~$A$ if it is finite, and instead, by a convenient abuse of notation, its Lebesgue measure if~$A \subset \Rd$ if it is a set in the continuum. 
We use normalized norms. To be precise, for a function~$f : \eta_*(U) \to \R$ we write 
\begin{equation*}
	\|f\|_{\underline{L}^2(\eta_*(U))} := \biggl( \avsum_{x \in \eta_*(U)} f(x)^2 \biggr)^{\sfrac12} \,. 
\end{equation*} 
Given~$x, y \in \eta_*$ we write~$x \sim y$ if~$|x-y|\leqslant 1$ (in Euclidean distance). We write 
\begin{equation*}
	\|f\|_{\underline{H}^1(\eta_*(U))} := \biggl(\avsum_{x , y : y \sim x\in \eta_*(U)} (f(x) - f(y))^2 \biggr)^{\sfrac12}\,. 
\end{equation*}
It is clear that~$\underline{H}^1(\eta_*(U))$ seminorm comes with a natural inner product: for functions~$f,g : \eta_*(U) \to \R$ we define 
\begin{equation*}
	\langle f, g \rangle_{\underline{H}^1(\eta_*(U))} := \avsum_{x, y  \in \eta_*(U) : y \sim x} \bigl( f(x) - f(y) \bigr) \bigl(g(x) - g(y) \bigr)\,. 
\end{equation*}
We define the~$\underline{H}^{-1}$ norm by duality via 
\begin{equation*}
	\|f\|_{\underline{H}^{-1}(\eta_*(U_m))} := \sup \biggl\{ \frac{\langle f, g \rangle_{\underline{H}^1(\eta*(U_m))}}{\|g\|_{\underline{H}^1(\eta_*(U_m))}} \bigg\vert g : \eta_*(U_m) \to \R, g(x) = 0 \mbox{ if } \mathrm{dist}(x,\partial U) \leqslant 2\biggr\}\,.
\end{equation*}

\subsection{Proof of Theorems~\ref{t.main-spec} and~\ref{c.LSR}} \label{ss.main}
	Throughout this section we think of~$m$ as fixed (but large, quantified by the statement of Theorem~\ref{t.main-spec}). 
	We begin by recalling some preliminaries and establishing the notation for this section. From the equation~\eqref{e.graphlap} we find that for any function~$u : \eta_*(U_m) \to \R$ we have 
	\begin{equation*}
		\avsum_{x \in \eta_*(U_m)} \sum_{y \in \eta_*(U_m)} \a(|y-x|) (u(y) - u(x)) u(x) = - \frac12 \avsum_{x,y \in U_m} \a(|y-x|)(u(y) - u(x))^2 < 0 \,,
	\end{equation*}
unless~$u$ is constant. Consequently, the map that takes~$f := - \nabla \cdot \ahom \nabla \uhom$ to~$u$ in~\eqref{e.GDP} can be represented by a symmetric matrix~$T_m$, that is positive definite, and has a list of positive eigenvalues, arranged in non-\emph{in}creasing order, labeled as
\begin{equation*}
	\mu_{m,1}  \geqslant \cdots \geqslant \mu_{m,N_m(U_0)} > 0, \ \  k \in \{1,2, \cdots, | \eta_*(U_m) \cap \{\mathrm{dist}(x,\partial U_m) \geqslant 2\}| \} =: N_m (U_0)\,.
\end{equation*}
We denote the associated eigenvectors by~$\{\phi_{m,k}\}_{k=1}^{N_m(U_0)},$ and we normalize these according to~\eqref{e.weirdnormalization}.  Setting 
\begin{equation*}
	\la_{m,k} := \mu_{m,k}^{-1}\,,
\end{equation*}
we have 
\begin{equation} \label{e.discreteeig}
	\begin{cases}
	-	\mathcal{L}\phi_{m,k} &= \la_{m,k} \phi_{m,k} \ \ \mbox{ in } \ \eta_* ( U_m) \cap \{\mathrm{dist}(x,\partial U_m) > 2\} \\
	\phi_{m,k} & = 0 \ \ \mbox{ on } \eta_* ( U_m) \cap \{\mathrm{dist}(x,\partial U_m) > 2\}\,. 
	\end{cases}
\end{equation}
The minimax theorem asserts that 
\begin{equation}
	\label{e.minmaxeps}
	\mu_{m,k} = \min_{f_1,\cdots, f_{k-1} } \biggl[\max_{\|f\|_{\underline{L}^2(U_m)} = 1, f \perp f_i, i = 1,\cdots, k-1}  \avsum_{x \in \eta_*(U_m)} T_m(f) (x) f(x)\biggr]\,;
\end{equation}
the first minimum is taken over functions~$f_1, \cdots, f_{k-1} : \mathrm{int}(\eta_*(U_m)) \to \R,$ and the orthogonality condition under the max is the condition that  for each~$i = 1,\cdots, k-1,$ admissible competitors must satisfy 
\begin{equation*}
	\sum_{x \in \eta_*(U_m)} f(x) f_i(x) = 0\,. 
\end{equation*}
Equivalently, setting 
\begin{equation*}
	V_{m,0} := \{0\}, \ \ V_{m,k} := \mathrm{span}( \phi_{m,1}, \cdots, \phi_{m,k-1}), \quad k \in \{1, \cdots, N_m(U_0)\}\,,
\end{equation*}
we have 
\begin{equation} \label{e.discrete-minmax}
	\mu_{m,k} = \max_{f \perp V_{m,k-1}, \|f \|_{\underline{L}^2(\eta_*(U_m)} = 1}  \avsum_{x \in \eta_*(U_m)} T_m (f)(x) f(x)\,.
\end{equation}

\smallskip 

Similarly, by standard elliptic theory and Theorem~\ref{t.AV2-main}, as~$\ahom(\la)$ is a positive, symmetric matrix, the operator~$-\nabla \cdot \ahom \nabla $ with Dirichlet boundary conditions \emph{on the domain~$U_0$}, has a list of eigenvalues~$\{\la_{0,k}\}_{k=1}^\infty$, arranged in nondecreasing order. For any~$f \in L^2(U_m),$ letting~$T_0(f) \in H^2(U_m) \cap H^1_0(U_m)$ be the unique solution to 
\begin{equation} \label{e.PDEhom}
	\begin{cases}
		- \nabla \cdot \ahom \nabla \uhom & = f, \ \ \mbox{ in } U_m\\
		\qquad \quad \ \ \uhom &= 0 \ \ \mbox{ on } \partial U_m\,, 
	\end{cases}
\end{equation} 
it is well known that~$T_0 : L^2(U_m) \to L^2(U_m)$ is a self-adjoint, compact operator. By scaling, the Dirichlet eigenvalues of~$-\nabla \cdot \ahom \nabla $ are given by~$3^{-2m}\la_{0,k}$, and so the associated eigenvalues of~$T_0$ are given by
\begin{equation} \label{e.mu0k}
\mu_{0,k} := 3^{2m}\la_{0,k}^{-1}\,.
\end{equation}
 We point out that, in our set up,~$\mu_{0,k}$ depends on~$m$ through a multiplicative factor, but we prefer not to explicitly notate this dependence, for clarity of exposition. Similarly to the discrete set up, we let~$\{\phi_{0,k}\}_{k \in \N}$ denote the orthonormal basis of~$L^2(U_m)$ of eigenfunctions of the Dirichlet Laplacian associated to~$-\nabla \cdot \ahom \nabla,$ \emph{on the domain~$U_m$}, and we set
\[
V_{0,0} = \{0\}, \ \ V_{0,k} := \mathrm{span} ( \phi_{0,1}, \cdots, \phi_{0,k-1}), \ \ k \in \N\,. 
\]
Analogous to~\eqref{e.discrete-minmax}, Courant's minmax theorem asserts that 
\begin{equation}
	\label{e.minmaxlimit}
	\begin{aligned}
	\mu_{0,k} &= \min_{f_1,\cdots f_{k-1} \in L^2(U_m)}\biggl[ \max_{ \|f\|_{\underline{L}^2(U_m)} = 1, \int_{U_m} f f_i = 0} \fint_{U_m} T_0(f) f\,dx\biggr]\\ 
&= \max_{f \perp V_{0,k-1}, \|f\|_{\underline{L}^2(U_m) = 1}} \fint_{U_m} T_0 (f)f\,dx\,.
	\end{aligned}
\end{equation}
By density, we may (and will) assume that the admissible competitors are continuous functions that are~$\underline{L}^2$ normalized-- this will enable us to restrict the function to the graph. 

\smallskip 
\begin{proof}[Proof of Theorem~\ref{t.main-spec}]
	The proof of~\eqref{e.eigvalrates}--\eqref{e.evratesL2} will be by induction on~$k$ that satisfy~\eqref{e.epskcondition}. It is organized as 
	\begin{itemize}
		\item a long \emph{Step 1}: the base case,
		\item a short \emph{Step 2}, which sets the induction hypothesis, and 
		\item the final \emph{Step 3}, which consists of the induction step. 
	\end{itemize}
	
	\smallskip 
	\emph{Step 1. Base case.} In this step we prove the desired rates when~$k = 1,$ i.e., for the ground states. This is the longest step in the proof and contains many of the ideas involved. 
	
	\emph{Convergence rates for the principal eigenvalue:}  By the variational principle for~$\mu_{0,1}$ we find that for any~$f \in C \cap L^2(U_m)$ with~$\|f\|_{\underline{L}^2(U_m)} = 1,$ we have 
	\begin{equation} \label{e.mu01}
		\mu_{0,1} = \max_{f \in C \cap L^2(U_m), \|f\|_{\underline{L}^2(U_m)} = 1} \fint_{U_m} T_0(f)f\,dx\,.
	\end{equation}
	 The maximizer is, of course, simple, and is given by~$f = \phi_{0,1}, $ which satisfies
	 
	 \[\fint_{U_m} T_0 (\phi_{0,1})\phi_{0,1} = \fint_{U_m} |\nabla \phi_{0,1}|^2 = 3^{-2m} \la_{0,1}\,.\]
	 In addition, by Lemma~\ref{l.moseriteration} we have~$\|\phi_{0,1}\|_{L^\infty(U_m)} \leqslant C_1(U_0) \la_{0,1}^{\frac{d}4}\,. $ 
	 Consequently, without loss of generality, we may restrict the competitors in the variational principle for~$\mu_{0,1},$ (namely,~\eqref{e.mu01}) to those functions~$f \in C \cap L^2(U_m)$, that satisfy in addition, that 
	 \begin{equation} \label{e.4times}
	 \|f\|_{L^\infty(U_m)} \leqslant 4\|\phi_{0,1}\|_{L^\infty(U_m)} \leqslant 4\la_{0,1}^{\frac{d}4}, \qquad 	\fint_{U_m} |\nabla f|^2 \leqslant 4 \cdot 3^{-2m} \la_{0,1}\,.
	 \end{equation}
 Before continuing with the argument, we record some basic estimates. Thanks to~\eqref{e.mu01}, for any admissible~$f,$ writing~$u_0 = T_0(f)$ we find 
 \begin{equation}
 	\label{e.basicest-1}
 	\fint_{U_m} T_0(f) f\,dx = \fint_{U_m} u_0 f\,dx \leqslant \mu_{0,1}\,,
 \end{equation}
so that, thanks to the ellipticity of~$\ahom,$ 
\begin{equation} \label{e.basicest-2}
	\|\nabla u_0\|_{\underline{L}^2(U_m)}^2 \leqslant C(d,\alpha) \mu_{0,1}\,,
\end{equation}
and thus
\begin{equation} \label{e.basicest-3}
	\|f\|_{\underline{H}^{-1}(U_m)} \leqslant C \|\nabla u_0\|_{\underline{L}^2 (U_m)} \leqslant C \sqrt{\mu_{0,1}}\,. 
\end{equation}
 Returning to the base case estimate, by the triangle inequality, we find 
	\begin{equation} \label{e.triangleineq-1}
		\begin{aligned}
			\mu_{0,1}  = \max \fint_{U_m} T_0(f) f\,dx &\leqslant \max \biggl[ \fint_{U_m} T_0(f) f\,dx - \frac{\avsum_{x \in \eta_*(U_m)} T_m(f) f }{\avsum_{x \in \eta_*(U_m)} f(x)^2}\biggr] \\
			&\qquad + \max \frac{\avsum_{x \in \eta_*(U_m)} T_m(f) f}{\avsum_{x \in \eta_*(U_m)} f(x)^2}\,.
		\end{aligned}
	\end{equation}
Lemma~\ref{l.montecarlo}, and its proof assert that if the dimension~$d \geqslant 3,$ then there is a universal constant~$C(d,U_0) > 0$ such that 
\begin{equation*}
\P \biggl[ \biggl| \avsum_{x \in \eta_*(U_m)} f(x)^2  - 1 \biggr| \geqslant 3^{-m} \sqrt{\la_{0,k}}\biggr] \leqslant  2 \exp \Bigl( -C(d,U_0) \bigl(3^m \la_{0,k}^{-\sfrac12}\bigr)^{d-2} \Bigr)\,,
\end{equation*}
whereas if~$d=2,$ then 
\begin{equation*}
		\P \biggl[ \biggl| \avsum_{x \in \eta_*(U_m)} f(x)^2  - 1 \biggr| \geqslant 3^{-m} \sqrt{m}\la_{0,k}^{\sfrac12}\biggr] \leqslant  2 \exp ( -C(U_0)m)\,.  
\end{equation*}
Therefore with probability exponentially close to~$1$ as indicated in the preceding two estimates, continuing from~\eqref{e.triangleineq-1}, we obtain 
\begin{equation} \label{e.onedirection}
	\mu_{0,1} - \mu_{m,1} \leqslant \max \biggl[ \fint_{U_m} T_0(f) f\,dx - \!\!\!\avsum_{x \in \eta_*(U_m)} T_m(f) f \biggr] + \mu_{m,1}3^{-m} \sqrt{\la_{0,1}} \bigl( \indc_{d \geqslant 3} + \sqrt{m} \indc_{d = 2}  \bigr)\,. 
\end{equation}
It is clear that the last term is~$\approx \sqrt{\mu_{0,1}},$ and so, we focus on the first, max term. For this term, setting~$u_m := T_m(f),$ and~$u_0 := T_0(f),$ we find 
\begin{equation} \label{e.task1.1}
	\begin{aligned}
		&\avsum_{x \in \eta_*(U_m)} u_m(x) f(x) - \fint_{U_m} T_0(f)(x) f(x)\,dx \\
		& \quad =  \avsum_{x \in \eta_*(U_m)}\bigl( u_m(x) - u_0(x) \bigr) f(x) + \avsum_{x \in \eta_*(U_m)} u_0(x) f(x)   - \fint_{U_m} u_0(x) f(x) \,dx \\
		&\quad =: \mathrm{I} + \mathrm{II}\,.
	\end{aligned}
\end{equation}
For the term~$\mathrm{I}$ we have,
	\begin{align}
	 \label{e.I1}
|\mathrm{I} |
= 
\biggl| \avsum_{x \in \eta_*(U_m)} \!\!\bigl( u_m(x) - u_0(x) \bigr) f(x)\biggr| 
& \leqslant \biggl( \avsum_{x \in \eta_*(U_m)} (u_m(x) - u_0(x))^2\biggr)^{\!\sfrac12 } \| f\|_{\underline{L}^2(\eta_*(U_m))}
\notag \\ &
\leqslant C 3^{-m}\|u_0\|_{\underline{L}^2(U_m)} \notag \\ &
\leqslant C\|\nabla u_0\|_{L^2(U_m)}  \leqslant C \sqrt{\mu_{0,1}}\,,
	\end{align}
where, we used the (continuum) Poincar\'e inequality in the last step. For term~$\mathrm{II},$ thanks to Lemma~\ref{l.MC2}, and~\eqref{e.basicest-3} and~\eqref{e.4times}, we find 
\begin{equation} \label{e.I2}
	\biggl| \avsum_{x \in \eta_*(U_m)} u_0(x) f(x)   - \fint_{U_m} u_0(x) f(x) \,dx \biggr| \leqslant C(1 + \sqrt{m}\indc_{d=2}) \sqrt{\mu_{0,1}}\,,  
\end{equation}
with probability exponentially close to~$1$ as quantified by Lemma~\ref{l.MC2}. Inserting~\eqref{e.task1.1},~\eqref{e.I1} and~\eqref{e.I2} into~\eqref{e.onedirection} yields
\begin{equation} \label{e.onedirection-final}
	\mu_{0,1} - \mu_{m,1} \leqslant C \sqrt{\mu_{0,1}}\,.
\end{equation}

\smallskip 

We turn to the opposite direction, still arguing towards the base case. To this goal we begin with the variational characterization of~$\mu_{m,1}.$ We have,
\begin{equation} \label{e.oppdir-1}
		\mu_{m,1} = \max_{f \in C \cap L^2(U_m): \|f\|_{\underline{L}^2(U_m) = 1}}  \frac{ \avsum_{x \in \eta_*(U_m)} T_m(f)(x) f(x) }{\avsum_{x \in \eta*(U_m)} f(x)^2}\,. 
\end{equation}
For technical reasons (we must estimate certain Monte-Carlo sums!) we introduce the auxiliary quantity 
\begin{equation} \label{e.oppdir-2}
	\mu_{m,1}^T :=  \max_{f \in \mathfrak{A}_T }  \frac{ \avsum_{x \in \eta_*(U_m)} T_m(f)(x) f(x) }{\avsum_{x \in \eta*(U_m)} f(x)^2}\,, 
\end{equation}
with 
\begin{equation} \label{e.ARdef}
	\mathfrak{A}_T := \{ f \in C \cap L^2(U_m): \|f\|_{\underline{L}^2(U_m) = 1}, \|f\|_{L^\infty(U_m)} \leqslant 4T \la_{0,1}^{\sfrac{d}4} \}\,.
\end{equation}
It is clear that the quantity~$\mu_{m,1}^T$ is independent of~$T$ for~$T$ sufficiently large (but possibly depending on~$m$, for now; we will show that we can take~$T$ independent of~$m$). In particular we notice that for each fixed~$m$ the quantity~$\mu_{m,1}^T$ is nondecreasing in~$T$ and satisfies
\begin{equation} \label{e.muT}
\max_T	\mu_{m,1}^T =  \mu_{m,1}\,. 
\end{equation}
We have, 
\begin{align*}
		\mu_{m,1}^T & \leqslant \max_{f \in \mathfrak{A}_T} \biggl[  \frac{ \avsum_{x \in \eta_*(U_m)} T_m(f)(x) f(x) }{\avsum_{x \in \eta*(U_m)} f(x)^2} - \fint_{U_m}T_0(f) f\,dx\biggr] \\ & \qquad
		+ \max_{f \in C \cap L^2(U_m): \|f\|_{\underline{L}^2(U_m) = 1}} \fint_{U_m}T_0(f) f\,dx\,,
 \end{align*} 
where we notice carefully, that the second~$\max$ on the right hand side of the preceding display can be made independent of~$T,$ and so, this maximum is, in fact, equal to~$\mu_{0,1}.$ This implies that 
\begin{equation*}
	\mu_{m,1}^T \leqslant \max_{f \in \mathfrak{A}_T} \biggl[  \frac{ \avsum_{x \in \eta_*(U_m)} T_m(f)(x) f(x) }{\avsum_{x \in \eta*(U_m)} f(x)^2} - \fint_{U_m}T_0(f) f\,dx\biggr] + \mu_{0,1}\,.
\end{equation*} 
But the arguments proving~\eqref{e.onedirection}--\eqref{e.I2} yield that we have 
\begin{equation*}
	\mu_{m,1}^T \leqslant C \sqrt{\mu_{0,1}} (1 + \sqrt{m}\indc_{d=2}) + \mu_{0,1}\,,
\end{equation*}
with probability at least~$1 - 2\exp \bigl( - \frac{C(d,U_0)(3^m \la_{0,1}^{-\sfrac12})^{d-2}}{T^2} \bigr)$ (if~$d \geqslant 3,$ and obvious modifications if~$d=2$).  As the right hand side is independent of~$T,$ we can maximize the left-hand side over~$T$ and use~\eqref{e.muT} to obtain 
\begin{equation} \label{e.otherdirection}
	\mu_{m,1} - \mu_{0,1} \leqslant C \sqrt{\mu_{0,1}} (1 + \sqrt{m}\indc_{d=2})\,.
\end{equation}
Combining~\eqref{e.onedirection-final} and~\eqref{e.otherdirection} we find that 
\begin{equation*}
	|\mu_{m,1} - \mu_{0,1}| \leqslant C \sqrt{\mu_{0,1}}(1 + \sqrt{m}\indc_{d=2})\,,
\end{equation*}
with probability at least~$1 -  2\exp \bigl( - \frac{C(d,U_0)(3^m \la_{0,1}^{-\sfrac12})^{d-2}}{T^2} \bigr)$ when~$d \geqslant 3$ and~$1 - 2\exp \bigl( - \frac{C(U_0) m}{T^2} \bigr)\,,$ when~$d=2.$ 
Since~$\mu_{m,1} = \frac{1}{\la_{m,1}},$ and~$\mu_{0,1} = \frac{3^{2m}}{\la_{0,1}},$ this yields that 
\begin{equation}
	\label{e.eigenvaluerates}
	|\la_{0,1} - 3^{2m}\la_{m,1}| \leqslant C 3^{-m} \la_{0,1}^{\sfrac32} \bigl( 1 + \sqrt{m}\indc_{d=2} \bigr)\,. 
\end{equation} 
This completes the proof of convergence rates of eigenvalues, \emph{provided we make a choice of~$T$ independent of~$m.$ }

\smallskip 
\emph{Convergence rates for the principal eigenvector.} We will show 
\begin{equation} \label{e.L2convrate-1}
	\|\phi_{m,k} - \phi_{0,k} \|_{\underline{L}^2(\eta_*(U_m))} \leqslant \frac{C}{\gamma(\la_{0,1})} 3^{-m} \la_{0,1}^{\sfrac32}  \bigl( 1 + \sqrt{m}\indc_{d=2} \bigr)\,,
\end{equation}
Here,~$\gamma(\la_{0,1})$ denotes the spectral gap of~$\la_{0,1},$ i.e., 
\begin{equation*}
	\gamma(\la_{0,1}) = |\la_{0,2} - \la_{0,1}| > 0. 
\end{equation*}
The proof of~$L^2$ convergence of eigenvectors relies on a two-scale convergence argument and follows closely the steps in~\cite[Section 5]{AV2}. Rather than repeat all the details, here we only offer a sketch of the proof, along with detailed pointers to relevant estimates in~\cite{AV2}. 
\begin{enumerate}
	\item We know that~$\phi_{m,1}$ solves~\eqref{e.discreteeig} corresponding to~$k=1,$ and~$\phi_{0,1}$ is the Dirichlet eigenfunction of~$-\nabla \cdot \ahom \nabla $ in~$U_m$ with eigenvalue~$\la_{0,1}.$ One constructs the two-scale approximation of~$\phi_{m,1}$ given by~$\phi_{0,1} + \sum_{i=1}^d \partial_{x_i} \phi_{0,1} \varphi_{e_i} (x),$ where~$\{\varphi_e:e\in \Rd\}$ denote the corrector fields constructed in~\cite[Section 7]{AV2}. 
	\item Defining
	\begin{equation*}
		w_{m,1} := \phi_{m,1} - \biggl(\phi_{0,1} + \sum_{i=1}^d \partial_{x_i} \phi_{0,1} \varphi_{e_i} (x) \biggr)\,,
	\end{equation*}
one computes 
\begin{multline}
\bigr(	\mathcal{L} - \la_{m,1} \bigl) w_{m,1} = - \mathcal{L}  \biggl(\phi_{0,1} + \sum_{i=1}^d \partial_{x_i} \phi_{0,1} \varphi_{e_i} (x) \biggr) + \nabla \cdot \ahom \nabla \phi_{0,1} \\
+ (3^{-2m}\la_{0,1} - \la_{m,1})\phi_{0,1} + \la_{m,1} \sum_{i=1}^d \partial_{x_i} \phi_{0,1} \varphi_{e_i}(x)\,. \label{e.ansatzeq}
\end{multline}
The details of this computation appear in~\cite[Lemmas 5.4 and 5.5]{AV2}. The desired estimate~\eqref{e.L2convrate-1} will follow by testing this equation with~$w_{m,1},$ provided we can control the right-hand side. 
\begin{enumerate}
	\item The first term is the two-scale expansion error term, and consists in measuring~$- \mathcal{L}  \biggl(\phi_{0,1} + \sum_{i=1}^d \partial_{x_i} \phi_{0,1} \varphi_{e_i} (x) \biggr) + \nabla \cdot \ahom \nabla \phi_{0,1}.$ The estimate for this difference in~$\underline{H}^{-1}(\eta_*(U_m))$ is contained in~\cite[Eq. (5.7)]{AV2}, and leads to the bound
	\begin{equation} \label{e.errorterm}
		C3^{-m} \sup_{|e|=1} \|\varphi_e\|_{\underline{L}^2(\eta_*(U_m))} \biggl(\avsum_{x \in \eta_*(U_m)} \sup_{y \in B_1(x)}|D\phi_{0,1}(y)|^2 + |D^2\phi_{0,1}(y)|^2 \biggr)^{\sfrac12}.
	\end{equation}
Optimal bounds on the correctors are in~\cite[Section 7, Eq (7.8)]{AV2}. They are essentially bounded, with an additional intrinsic~$\sqrt{m}$ in two dimensions. In order to control the last factor, we use \emph{interior regularity of the continuum eigenfunctions~$\phi_{0,1}.$ } Indeed, for a fixed center~$x$ in the graph we consider the change of variables~$z = \frac{y - x}{\sqrt{\la_{0,1}}},$ which changes the domain~$U_m$ into the ~$\frac{1}{\sqrt{\la_{0,1}}} (U_m - x).$ By~\eqref{e.epskcondition} this domain is ``still large'', and under this change of variables, defining~$\overline{\phi_{0,1}}(z) := \phi_{0,1}(\sqrt{\la_{0,1}} z + x)$ entails that~$\overline{\phi_{0,1}}(z)$ solves~\[-\nabla \cdot\ahom \nabla \overline{\phi_{0,1}}(z)= \overline{\phi_{0,1}}(z)\,.\]
As~$\ahom$ is constant, interior estimates then imply that the Hessian of~$\overline{\phi_{0,1}}$ can be controlled by its~$L^2$ norm (which is independent of~$\la_{0,1}$). Scaling back to the original variables, this implies that the last factor in~\eqref{e.errorterm} is controlled by~$\la_{0,1}^{\sfrac12}.$  Explicitly: writing~$\la \equiv \la_{0,1}$,
\begin{multline*}
	\|D^2\phi_{0,1}\|_{L^\infty (B_1)} = \la \|D^2\overline{\phi_{0,1}}\|_{L^\infty\bigl(B_{\la^{-\sfrac12}}\bigr)} \\ \leqslant C \la \|\overline{\phi_{0,1}}\|_{L^2\bigl(B_{2\la^{-\sfrac12}}\bigr)} = C \la^{1-\frac{d}4} \|\phi_{0,1}\|_{L^2(B_1)} \,. 
\end{multline*}
\item The second term~$|\la_{0,1} - \la_{m,1}|$ is controlled, thanks to~\eqref{e.eigenvaluerates} and Lemma~\ref{l.montecarlo}. 
\item Finally, in order to control the last term, we use that~$\la_{m,1} \approx 3^{-2m}\la_{0,1}$ by~\eqref{e.eigenvaluerates}, and estimate 
\begin{multline*}
\biggl| 3^{-m}	\avsum_{x \in \eta_*(U_m)} \partial_{x_i} \phi_{0,1}(x) \varphi_{e_i}(x) w_{m,1}(x) \biggr| \\
\leqslant 3^{-m}\|\varphi_{e_i}\|_{\underline{H}^{-1}(U_m)} \biggl( \avsum_{x, y \sim x \in \eta_*(U_m)} \bigl( (\partial_{x_i}  \phi_{0,1} w_{m,1})(x)  - (\partial_{x_i} \phi_{0,1} w_{m,1})(y)\bigr)^2\biggr)^{\sfrac12}\,. 
\end{multline*}
For the corrector term~$3^{-m}\|\varphi_{e_i}\|_{\underline{H}^{-1}(U_m)}$ we use the optimal corrector estimates from~\cite[Theorem 1.2, Eq (1.6)]{AV2} with the choice~$\mathfrak{s} = 1$ in the notation there, we find that this factor is of size~$3^{-m}.$ For the other factor, using the discrete product rule in the form 
\begin{equation*}
	(fg)(x) - (fg)(y) = \frac12 (f(x) + f(y)) (g(x) - g(y)) + \frac12 (g(x) + g(y)) (f(x) - f(y))\,,
\end{equation*}
followed by Cauchy-Schwarz, and repeating the arguments from (a) to estimate~\eqref{e.errorterm} completes the desired estimate. 
\end{enumerate}
This completes the argument for~\eqref{e.L2convrate-1}. 
\end{enumerate}

Finally, we prove the~$L^\infty$ convergence rates for the eigenfunctions, which will enable selecting the parameter~$T$ in~\eqref{e.ARdef}. 

\smallskip  
Recall from~\eqref{e.ansatzeq} that the function~$w_{m,1}$ above solves an equation of the form~$\bigl(\mathcal{L} - \la_{m,1} \bigr) w_{m,1} = f.$  Now, let~$y \in U_0$ be such that~$x := 3^m y \in \eta_*(U_m),$ and let
\begin{equation*}
R := \frac12 3^m \mathrm{dist}(y, \partial U_0)\,.
\end{equation*}
 By the large-scale regularity theorem~\cite[Theorem 1.1]{AV2}, there exists~$C(d,\la) < \infty$ and~$s(d) > 0$ and a nonnegative random variable~$\X$ satisfying 
\begin{equation*}
	\X = \O_s(C)
\end{equation*}
such that if~$R \geqslant \X,$ then for any~$r \in [\X, \frac12 R]$ we have the estimate 
\begin{multline}
	\label{e.LSR-eigfct}
\sup_{t \in [r,\frac12R]} \|\nabla w_{m,1}\|_{\underline{L}^2(\eta_*(B_t(x))} \leqslant \frac{C}{\mathrm{dist}(y,\partial U_0)} 3^{-m}  \| w_{m,1} - (w_{m,1})_{B_R(x)} \|_{\underline{L}^2(\eta_*(B_R(x)))} \\ 
+ C \int_r^R \|f\|_{\underline{H}^{-1}(\eta_*(B_t(x)))} \frac{\,dt }{t}\,.
\end{multline}
We must simplify the last term of~\eqref{e.LSR-eigfct}, which requires estimating~$\|f\|_{\underline{H}^{-1}(\eta_*(B_t))}.$ By the triangle inequality applied to the right-hand side in~\eqref{e.ansatzeq}, we have 
\begin{multline*}
		\|f\|_{\underline{H}^{-1} (\eta_*(B_t))} \leqslant \biggl\| - \mathcal{L}  \biggl(\phi_{0,1} + \sum_{i=1}^d \partial_{x_i} \phi_{0,1} \varphi_{e_i} (x) \biggr) + \nabla \cdot \ahom \nabla \phi_{0,1} \biggr\|_{\underline{H}^{-1}(\eta_*(B_t))} \\
		+|3^{-2m} \la_{0,1} - \la_{m,1} \|\phi_{0,1} \|_{\underline{H}^{-1}(\eta_*(B_t))} + \la_{m,1} \sum_{i=1}^d\biggl\|  \partial_{x_i} \phi_{0,1} \varphi_{e_i} \biggr\|_{\underline{H}^{-1}(\eta_*(B_t))} = : A + B + C\,.
\end{multline*}
We start with estimating~$B$. By using the definition of the discrete~$H^{-1}$ norm, then the Poincar\'e inequality (see \cite[Lemma 2.3]{AV2}) and Lemma~\ref{l.moseriteration} we obtain by Cauchy-Schwarz that
\begin{multline*}
	\| \phi_{0,1}\|_{\underline{H}^{-1}(\eta_*(B_t))} = \sup_{g \in H^1_0(\eta_*(B_t)) : \|g\|_{\underline{H}^1(\eta_*(B_t))} \leqslant 1} \avsum_{x \in \eta_*(B_t)} \phi_{0,1}(x) g(x) \\
	\leqslant C \la_{0,1}^{\sfrac{d}4} \sup_g \biggl( \avsum_{x \in \eta_*(B_t)} g(x)^2\biggr)^{\sfrac12} \leqslant C \la_{0,1}^{\sfrac{d}4} t \,.
\end{multline*}
Therefore, 
\begin{equation}
	\label{e.estforB}
	|3^{-2m} \la_{0,1} - \la_{m,1}| \|\phi_{0,1}\|_{\underline{H}^{-1}(\eta_*(B_t))} \leqslant C3^{-3m}\la_{0,1}^{\sfrac32+\sfrac{d}4}(1 + \sqrt{m}\indc_{d=2})t\,. 
\end{equation}
Turning to the term~$C$ we have, for each~$i \in \{1,\cdots, d\},$ thanks to Poincar\`e inequality, and~\cite[Theorem 1.2]{AV2} on the optimal sizes of the correctors, 
\begin{equation} \label{e.C}
	\begin{aligned}
\la_{m,1}	\biggl\|  \partial_{x_i} \phi_{0,1}\varphi_{e_i}(x)\biggr\|_{\underline{H}^{-1}(\eta_*(B_t))} &\leqslant 2 \cdot 3^{-2m} \la_{0,1} \sup_g \avsum_{x \in \eta_*(B_t)} \partial_{x_i} \phi_{0,1} \varphi_{e_i}(x) g(x) \\
&\leqslant  2 \cdot 3^{-2m} \la_{0,1}  \|\varphi_{e_i}\|_{\underline{L}^{2}(\eta_*(B_t))} \sup_g \| \partial_{x_i} \phi_{0,1} g \|_{\underline{L}^2(\eta_*(B_t))}\\
&\leqslant C 3^{-2m} (\indc_{d \geqslant 3} + \sqrt{m} \indc_{d=2}) \la_{0,1}^{\frac{d}4 + \frac32} t\,.
	\end{aligned}
\end{equation}
Finally we estimate term~$A.$ For this term, arguing as in~\cite[Lemmas 5.4--5.5]{AV2} and the optimal bounds for correctors from~\cite[Theorem 1.2]{AV2} we obtain 
\begin{equation} \label{e.A}
		t^{-1}A \leqslant C (1 + \sqrt{m}\indc_{d=2}) 3^{-2m} \la_{0,1}^{\sfrac32}\,. 
\end{equation}
It follows by plugging in~\eqref{e.estforB},\eqref{e.C}, and \eqref{e.A} into~\eqref{e.LSR-eigfct} that 
\begin{equation} \label{e.LSR-final}
	\begin{aligned}
\sup_{t \in [r,\frac12 R]} \| \nabla w_{m,1}\|_{\underline{L}^2(\eta_*(B_t(x)))} &\leqslant \frac{C}{\mathrm{dist}(y,\partial U_0)} 3^{-m} \| w_{m,1}  - (w_{m,1})\|_{\underline{L}^2(\eta_*(B_R(x)))} \\
&\quad + C(1 + \sqrt{m}\indc_{d=2})  \la_{0,1}^{\sfrac32 + \sfrac{d}4}3^{-m}\\
&\leqslant \frac{C}{\mathrm{dist}(y,\partial U_0)} (1 +  \sqrt{m} \indc_{d=2})\la_{0,1}^{\sfrac{d}{4} + \sfrac{3}{2}}  3^{-m}\,.
	\end{aligned}
\end{equation}
Since any~$y,z \in \eta_*\bigl(B_{\frac12 \X}(x)\bigr)$ can be joined by a path with no more than~$C |B_{\X} (x)| = C \X^{d}$ edges, by discreteness, the triangle inequality, and Cauchy-Schwarz we find
\begin{align*}
	|w_{m,1}(y) - w_{m,1}(z)| & \leqslant 
C \X^{\sfrac{d}2} \|\nabla w_{m,1}\|_{\underline{L}^2(\eta_*(B_t))} 	\\ & 
	 \leqslant \frac{C}{\mathrm{dist}(y,U_0)}\X^{\sfrac{d}2} (1 +  \sqrt{m} \indc_{d=2}) \la_{0,1}^{\sfrac{d}{4} + \sfrac{3}{2}} 3^{-m}\,.
\end{align*}
As~$\X^{\sfrac{d}2} \leqslant \O_s(C),$ the proof of assertion~\eqref{e.C01} is complete for the case~$k=1$. 
Finally in order to get an~$L^\infty$ bound, we note that for any~$\zeta = 3^m y \in \eta_*(U_m), y \in U_0,$ by the triangle inequality, Poincar\'{e} inequality on the ball~$B_{\sfrac12\X(\zeta)}(\zeta)$ and estimates~\eqref{e.LSR-final}, and~\eqref{e.L2convrate-1}, we obtain 
\begin{equation*}
	\begin{aligned}
	w_{m,1}^2(\zeta) &\leqslant C\X^2(\zeta) \avsum_{x \in B_{\frac12 \X(\zeta)}} w_{m,1}^2(x) \\
&\leqslant C \X^4 (\zeta) \biggl\{ \avsum_{x \in B_{\frac12 \X(\zeta)}}\bigl( w_{m,1} {-} (w_{m,1})_{B_{\frac12 \X(\zeta)}}\bigr)^2 \biggr\} + C \X^2(\zeta) \bigl| (w_{m,1})_{B_{\frac12 \X(\zeta)}} {-} (w_{m,1})_{B_R(\zeta)} \bigr|^2 \\
&\quad \quad + C\X^2(\zeta) \|w_{m,1}\|_{\underline{L}^2(B_R(\zeta))}^2\\	
&\leqslant C \X^4(\zeta) (1 + \sqrt{m}\indc_{d=2}) \la_{0,1}^{\sfrac{d}2+ 3} 3^{-2m} \,. 
	\end{aligned}
\end{equation*}
This implies assertion~\eqref{e.Loo} for~$k=1.$ 
The upshot of the foregoing analysis and Lemma~\ref{l.moseriteration} is that for all~$m$ so that~$3^m\geqslant \X$ for which~\eqref{e.epskcondition} holds, we have  
\begin{equation*}
	\|\phi_{m,1}\|_{L^\infty(\eta_*(U_m))} \leqslant \|w_{m,1}\|_{L^\infty(\eta_*(U_m))} +  \|\phi_{0,1}\|_{L^\infty(U_m)} \leqslant 2 C_0(d) \la_{0,1}^{\frac{d}4}(1 + \la_{0,1})\,. 
\end{equation*}
It follows that in~\eqref{e.ARdef} we may choose~$T := 2C_0(d)(1 + \la_{0,1}) $. Having fixed the choice of~$T$ independently of~$m,$ the proof of the convergence rates of eigenvalues and eigenvectors hold, completing the proof of the base case. 

\bigskip 

\emph{Step 2.} \emph{Induction hypothesis:}
 Suppose that for some~$K \in \N$ with~$\la_{0,K}$ verifying~\eqref{e.epskcondition} we have shown convergence rates for eigenvalues~$\{\la_{m,j}\}_{j=1}^{K-1}$ and corresponding rates for eigenfunctions in~$L^2$ and~$L^\infty:$
 \begin{equation*}
 	\begin{aligned}
 	|\la_{0,j} - 3^{2m} \la_{m,j}| &\leqslant C 3^{-m} \la_{0,j}^{\sfrac32} \bigl( 1 + \sqrt{m}\indc_{d=2}\bigr)\\
 	\|\phi_{m,j} - \phi_{0,j} \|_{\underline{L}^2(\eta_*(U_m))} &\leqslant C3^{-m}\la_{0,j}^{\sfrac32} \frac{1}{\gamma(\la_{0,j})} \bigl( 1 + \sqrt{m}\indc_{d=2}\bigr)\,.
 	\end{aligned}
 \end{equation*}  
where
\begin{equation*}
	\gamma(\la_{0,j}) := \min\{|\la_{0,j} - \la_{0,\ell}| : \la_{0,\ell} \neq \la_{0,j}\}\,. 
\end{equation*}

\bigskip 
\emph{Step 3.} \emph{Induction step:}
To conclude the proof, we must show the rates for the~$K$th eigenvalue and eigenvector. In fact, once we have shown the convergence rates for eigenvalues, the argument for eigenvectors (in~$L^2$ and~$L^\infty$) proceed as in Step 1 with obvious notational changes. Therefore, we focus on proving the convergence rates for eigenvalues, and as in Step 1, prove that with high probability,
\[
|\mu_{m,K} - \mu_{0,K}| \leqslant C \sqrt{\mu_{0,K}}\,.
\]
We have,
\begin{equation} \label{e.mu0K}
	\mu_{0,K} = \min_{\begin{subarray}
		     	\, \Sigma \subset L^2(U_m) ,\\
			\mathrm{dim}\, \Sigma = K-1
		\end{subarray} } \max_{\begin{subarray}
	\,\,\,	f \in C \cap L^2(U_m),\\
		\|f\|_{\underline{L}^2(U_m)} = 1, f \perp \Sigma
	\end{subarray} } \fint_{U_m} T_0(f) f\,dx\,.
\end{equation}
Now we claim (\emph{using the induction hypothesis!}) that for each~$f \in C \cap L^2(U_m),$ there exists~$g \in C \cap L^2(U_m),$ with
\begin{equation} \label{e.indhyprate}
	\|g - f\|_{\underline{L}^2(\eta_*(U_m))} \leqslant C3^{-m}\max_{i\in \{1,\cdots, K-1\}}\frac{ \la_{0,i}^{\sfrac32}}{\gamma(\la_{0,i})} \underbrace{\leqslant}_{\eqref{e.evratesL2}} C\,, 
\end{equation}
which satisfies
\begin{multline} \label{e.iff}
		\avsum_{x \in \eta_*(U_m)} f(x) \phi_{m,i}(x) = 0  \ \ \forall i \in \{1,\cdots, K-1\}\ \ \\
	 \iff  \fint_{U_m} g(x) \phi_{0,i}(x) \,dx = 0 \ \ \forall i \in \{1,\cdots, K-1\}\,. 
\end{multline}
Granting the claim in~\eqref{e.indhyprate}--\eqref{e.iff} for the moment, and letting~$\Sigma_K \subset L^2(U_m)$ denote the span of the functions~$\{\phi_{0,i} : i \in \{1,\cdots, K-1\}\},$ we continue from~\eqref{e.mu0K} to obtain 
\begin{equation} \label{e.varpring0K}
\mu_{0,K} = \max_{\begin{subarray}
		\,\,\,	g \in C \cap L^2(U_m),\\
		\|g\|_{\underline{L}^2(U_m)} = 1, g \perp \Sigma_K
\end{subarray} } \fint_{U_m} T_0(g) g\,dx\,.
\end{equation}
The optimizer of this variational problem exists and is unique, and is achieved by~$g = \phi_{0,K}.$ Therefore, in the variational principle~\eqref{e.varpring0K} we may restrict the competitors to additionally satisfy (thanks to Proposition~\ref{l.moseriteration}):
\begin{equation}
	\label{e.additionally}
	\|g \|_{L^\infty(U_m)} \leqslant 4 \|\phi_{0,K}\|_{L^\infty(U_m)} \leqslant 4C_0(d,U_0)  3^{-2m} \la_{0,K}^{\sfrac{d}4}\,.
\end{equation}
We continue to estimate
\begin{equation} 
	\begin{aligned}
		&\mu_{0,K} = \max_{\begin{subarray}
				\,\,\,	g \in C \cap L^2(U_m),\\
				\|g\|_{\underline{L}^2(U_m)} = 1, g \perp \Sigma_K\\
				g \mbox{ \scriptsize verifies } \eqref{e.additionally}
		\end{subarray} } \fint_{U_m} T_0(g) g\,dx\\
	&\quad \quad  \leqslant \max_{\begin{subarray}
		\,\,\,	g \in C \cap L^2(U_m),\\
		\|g\|_{\underline{L}^2(U_m)} = 1, g \perp \Sigma_K\\
		g \mbox{ \scriptsize verifies } \eqref{e.additionally}
	\end{subarray} } \biggl\{\fint_{U_m} T_0(g) g\,dx - \avsum_{x \in \eta_*(U_m)} T_m(g) g \biggr\} \\
&\quad \quad \quad + \max_{\begin{subarray}
		\,\,\,	g \in C \cap L^2(U_m),\\
		\|g\|_{\underline{L}^2(U_m)} = 1, g \perp \Sigma_K\\
		g \mbox{ \scriptsize verifies } \eqref{e.additionally}
\end{subarray} } \avsum_{x \in \eta_*(U_m)} T_m(g) g\,,
	\end{aligned}\label{e.onewayindstep}
\end{equation}
The last term is essentially~$\mu_{m,K}.$ To see this, first we note that for \emph{any} admissible~$g,$ and corresponding~$f$ satisfying~\eqref{e.indhyprate}--\eqref{e.iff}, by the Poincar\'{e} inequality we have
\begin{multline}
	\biggl| \avsum_{x \in \eta_*(U_m)} T_m(f) f -  T_m(g)g \biggr| \leqslant 
	\|T_m (f) - T_m(g)\|_{\underline{H}^1(\eta_*(U_m))} \|f\|_{\underline{H}^{-1}(\eta_*(U_m))} \\
	+ \|T_m(g)\|_{\underline{H}^1(\eta_*(U_m))} \|f - g\|_{\underline{H}^{-1}(\eta_*(U_m))}
	\leqslant C\sqrt{\mu_{m,K}}\,,
\end{multline}
where in the last inequality we used the discrete versions of~\eqref{e.basicest-1}--\eqref{e.basicest-3}, together with the Poincar\'e inequality. Therefore it follows that 
\begin{equation}
\max_{\begin{subarray}
		\,\,\,	g \in C \cap L^2(U_m),\\
		\|g\|_{\underline{L}^2(U_m)} = 1, g \perp \Sigma_K\\
		g \mbox{ \scriptsize verifies } \eqref{e.additionally}
\end{subarray} } \avsum_{x \in \eta_*(U_m)} T_m(g) g \leqslant \mu_{m,K}+ C\sqrt{\mu_{m,K}}\,. 
\end{equation}
Rearranging, 
\begin{equation} \label{e.muKstep}
	\mu_{0,K} - \mu_{m,K} \leqslant  \max_{\begin{subarray}
			\,\,\,	f \in C \cap L^2(U_m),\\
			\|f\|_{\underline{L}^2(U_m)} = 1, f \perp \Sigma_K
	\end{subarray} } \biggl\{\fint_{U_m} T_0(f) f\,dx - \avsum_{x \in \eta_*(U_m)} T_m(f) f \biggr\} +C \sqrt{\mu_{m,K}}\,.
\end{equation}
The proof of the opposite direction is similar, and the remainder of the argument proceeds exactly as in the base case with minor changes. The proofs of Theorems~\ref{t.main-spec} and~\ref{c.LSR} are now complete. 
\end{proof}

\appendix
\section{Some auxiliary estimates used in the proof} \label{s.appendix}

 \begin{lemma}
 	\label{l.moseriteration}
 Let~$U_0 \subset \cu_0$ be a bounded, Lipschitz domain. There exists a constant~$C_0(d) > 0$ such that every solution of
\begin{equation*}
	\begin{cases}
		- \nabla \cdot \ahom \nabla \phi &= \la \phi \quad \mbox{ in } U_0\\
		\phi &= 0 \quad \mbox{ on } \partial U_0\,.
	\end{cases}
\end{equation*}
satisfies the estimate
\begin{equation*}
	\|\phi\|_{L^\infty(U_0)} \leqslant C_0(d) \la^{\sfrac{d}4} \|\phi\|_{L^2(U_0)}\,. 
\end{equation*}
 \end{lemma}
\begin{proof}
	See~\cite[Proposition 5]{MF}.
\end{proof}
\begin{lemma}
	\label{l.montecarlo}
	Let~$k \in \N$ satisfy~\eqref{e.epskcondition}. Then, there exists a universal constant~$C(d,U_0)< \infty$ such that if~$d \geqslant 3,$ then 
	\begin{equation} \label{e.MCdbig3}
		\P \biggl[ \biggl| \avsum_{x \in \eta_*(U_m)} \phi_{0,k}(x)^2  - 1 \biggr| \geqslant 3^{-m} \sqrt{\la_{0,k}}\biggr] \leqslant  2 \exp \Bigl( -C(d,U_0) \bigl(3^m \la_{0,k}^{-\sfrac12}\bigr)^{d-2} \Bigr)\,.  
	\end{equation}
Instead, if~$d=2,$ then there exists a constant~$C(U_0) <\infty$ 
	\begin{equation} \label{e.MCd2}
	\P \biggl[ \biggl| \avsum_{x \in \eta_*(U_m)} \phi_{0,k}(x)^2  - 1 \biggr| \geqslant 3^{-m} \sqrt{m}\la_{0,k}^{\sfrac12}\biggr] \leqslant  2 \exp ( -C(U_0)m)\,.  
\end{equation}
\end{lemma}
\begin{proof}
	Defining~$\psi_{0,k} \in L^2(U_0)$ by~$\psi_{0,k}(y):= \phi_{0,k}(3^m y),$
	\begin{equation*}
		I_{m} (\phi_{0,k}) := \avsum_{x \in \eta_*(U_m)} \phi_{0,k}(x)^2 = \avsum_{y \in 3^{-m} \eta_*(U_m) } \psi_{0,k}^2(y)\,,
	\end{equation*}
and 
\begin{equation*}
	I(\phi_{0,k}) := \fint_{U_m} \phi_{0,k}^2(x)\,dx = 1\,.
\end{equation*}
Obviously,~$\psi_{0,k}$ is a Dirichlet eigenfunction of~$-\nabla \cdot \ahom\nabla$ on the domain~$U_0$ with eigenvalue~$\la_{0,k}$. It is clear that~$I_m$ is a sum of~$\P-$independent, identically distributed random variables~$\{Z_y := \psi_{0,k}(y)^2: y \in 3^{-m}\eta_*(U_m)\}$. Thus Bernstein's inequality implies that for any~$t > 0$ we have
\begin{equation} \label{e.bernstein}
	\P\biggl[ |I_m(\phi_{0,k}) - 1| \geqslant t\biggr] \leqslant 2 \exp \biggl( - \frac{|\eta_*(U_m)|t^2}{2\sigma^2 } \biggr)\,,
\end{equation}
Here,~$\sigma^2$ is the variance of the random variables~$\{Z_y\},$ and computes as 
\begin{equation*}
	\sigma^2 = \fint_{U_0}  \psi_{0,k}^4(y)\,dy - \biggl(\fint_{U_0} \psi_{0,k}^2(y)\,dy\biggr)^2 \leqslant \|\psi_{0,k}\|_{L^\infty(U_0)}^2  - 1\,. 
\end{equation*}
Applying a Moser iteration, and using the PDE for~$\psi_{0,k},$ (see for example~\cite[Proposition 5]{MF}) yields that there is a dimensional constant~$C(d) > 0$ such that
\begin{equation*}
	\|\psi_{0,k}\|_{L^\infty(U_0)} \leqslant C(d) \la_{0,k}^{\frac{d}4} \|\psi_{0,k}\|_{\underline{L}^2(U_0)} \leqslant C(d) \la_{0,k}^{\frac{d}4}\,. 
\end{equation*}
Inserting this, and using the choice~$t := 3^{-m} \sqrt{\la_{0,k}}$ if~$d \geqslant 3,$ and~$t := 3^{-m}\sqrt{m} \la_{0,k}^{\sfrac12}$ if~$d=2$ yields the desired estimate. 
\end{proof}
For the next lemma, we let~$f \in C \cap L^2(U_m)$ with~$\|f\|_{\underline{L}^2(U_m)} = 1.$ We let~$u_0$ denote the solution to 
\begin{equation*}
	\begin{cases}
		-\nabla \cdot \ahom \nabla u_0 &= f \quad \mbox{ in } U_m\\
		u_0 &= 0 \quad \mbox{ on } \partial U_m\,,
	\end{cases}
\end{equation*}
so that we have~$u_0(x) = T_0(f).$ During the course of our arguments for the proof of convergence rates we had to estimate the convergence rate of the Monte-Carlo approximation of the integral 
\begin{equation*}
\biggl|	\avsum_{x \in \eta_*(U_m)} u_0(x) f(x) - \fint_{U_m} u_0(x) f(x)\,dx \biggr|\,.
\end{equation*}
The next lemma provides the desired convergence rate. 
\begin{lemma}
	\label{l.MC2}
There exists a constant~$C(U_0,d) < \infty$ such that the following holds: if the dimension~$d \geqslant 3,$ then
\begin{multline}
	\P \biggl[ \biggl|	\avsum_{x \in \eta_*(U_m)} u_0(x) f(x) - \fint_{U_m} u_0(x) f(x)\,dx \biggr| > \|f\|_{\underline{H}^{-1}(U_m)} \biggr] \\
	 \leqslant 2 \exp \biggl( -  \frac{3^{m(d-2)} \bigl( 3^{-m}\|f\|_{\underline{H}^{-1}(U_m)})^2}{C(U_0,d) \|f\|^2_{L^\infty(U_m)}}\biggr)\,.  
\end{multline}
Instead, if~$d = 2$ then 
\begin{multline} 
	\P \biggl[3^{-2m}\biggl|	\avsum_{x \in \eta_*(U_m)} u_0(x) f(x) - \fint_{U_m} u_0(x) f(x)\,dx \biggr| > \sqrt{m}\|f\|_{\underline{H}^{-1}(U_m)} \biggr] \\
	\leqslant 2 \exp \biggl( -  \frac{m (3^{-m}\|f\|_{\underline{H}^{-1}(U_m)})^2}{C(U_0,d) \|f\|^2_{L^\infty(U_m)}}\biggr)\,.  
\end{multline}
\end{lemma}
\begin{proof}
	We let~$F \in L^2(U_0) $ with~$\|F\|_{\underline{L}^2(U_0)} = 1$ be defined by 
	\begin{equation*}
		F(y) = f(3^m y)\,.
	\end{equation*}
Let~$v_0(y) := 3^{-2m} u_0(3^m y);$ then a simple  change of variables implies that we must estimate 
\begin{equation} \label{e.goal}
	\begin{aligned}
&	\biggl| \avsum_{x \in \eta_*(U_m) } u_0(x) f(x) - 3^{2m} \fint_{U_0} v_0(y) F(y)\,dy  \biggr|\\
	&\quad  = 3^{2m}\biggl| \avsum_{y \in 3^{-m}\eta_*(U_m)} v_0(y) F(y) - \fint_{U_0} v_0(y) F(y)\,dy \biggr|\,.  
	\end{aligned}
\end{equation} 
The proof of the desired estimate follows then the same argument as in Lemma~\ref{l.montecarlo} above. 
\end{proof}

\bibliographystyle{alpha}
\bibliography{ref.bib}
\end{document}